\documentclass[12pt,leqno]{amsart} 
\usepackage[numbers]{natbib} 
\usepackage{amsmath,amsthm}
\usepackage{amssymb}

\usepackage{url}
\usepackage[colorlinks={true},
backref,pdfstartview={FitBH},pdfview={FitBH},bookmarks={true}]{hyperref}
\global\hbadness=500
\global\tolerance=200
\newtheorem{theorem}{Theorem}[section]
\newtheorem{lemma}[theorem]{Lemma}

\newtheorem{question}[theorem]{Question}

\newtheorem{corollary}[theorem]{Corollary} 
\theoremstyle{definition}
\newtheorem{definition}[theorem]{Definition}

\theoremstyle{remark}

\numberwithin{equation}{theorem}

\renewcommand{\phi}{\varphi}

\title[Splits of c.e.\ sets]{On Splits of Computably enumerable sets}
  
\author[Peter Cholak]{Peter~A.~Cholak}

\address{Department of Mathematics\\ University of Notre Dame\\ 
  Notre Dame, IN 46556-5683}

\email{Peter.Cholak.1@nd.edu}

\urladdr{http://www.nd.edu/~cholak}

\thanks{Draft as of \today. We want to thank V.\ Yu.\ Shavrukov for
  allowing us to include his result, Theorem~\ref{Shav}.  Without it,
  this paper would look very different. This research was started
  while Cholak participated in the Buenos Aires Semester in
  Computability, Complexity and Randomness, 2013.  Thanks to Rachel
  Epstein, Greg Igusa, Nathan Pierson, Mike Stob, and the referees for
  comments and suggestions.  My interest in Friedberg splits was
  sparked in 1989 by Rod Downey. I cannot forgive him.}

\subjclass[2000]{Primary 03D25}

\begin{document}

 \begin{abstract}

   Our focus will be on the computably enumerable (c.e.) sets and
   trivial, non-trivial, Friedberg, and non-Friedberg splits of the
   c.e.\ sets.  Every non-computable set has a non-trivial Friedberg
   split.  Moreover, this theorem is uniform.  V.\ Yu.\ Shavrukov
   recently answered the question which c.e.\ sets have a non-trivial
   non-Friedberg splitting and we provide a different proof of his
   result.  We end by showing there is no uniform splitting of all
   c.e.\ sets such that all non-computable sets are non-trivially
   split and, in addition, all sets with a non-trivial non-Friedberg
   split are split accordingly.
 \end{abstract}

 \maketitle

 \section{Trivial Splits}

 Given a c.e.\ set $A$, a \emph{split} of $A$ is a pair of c.e.\ sets
 $A_0, A_1$ such that $A_0 \sqcup A_1 = A$, $\sqcup$ is disjoint
 union.  If one of $A_0$ or $A_1$ is computable the splitting is
 \emph{trivial}.  If $A_0$ is computable then
 $A = A_0 \sqcup (\overline{A_0} \cap A)$.

 It is straightforward to see that any splitting of a computable set
 is trivial.  Given a c.e.\ set $A$, letting $A_0 = \emptyset$ and
 $A_1 = A$, provides a trivial splitting of $A$.  We would like to
 avoid splits where one of the sets is finite.  It is known that every
 infinite c.e.\ set $A$ has an infinite computable subset $R$.  This
 provides a trivial splitting of $A$,
 $A = R \sqcup (A \cap \overline{R})$, into two infinite c.e.\ sets
 assuming $A$ is not computable. 

Given this, Myhill asked 
\begin{question}[\citet{Myhill}]
  Does every non-computable c.e.\ set have a non-trivial splitting?
\end{question}
Myhill's Question was answered positively by \citet{MR0109125}.

\section{Friedberg Splits}

Most of this section is known but we wanted to provide an explicit
proof of Corollary~\ref{uni}.  This corollary will be useful
later.  One of the focuses of this paper is splitting procedures that
always produce a non-trivial split when possible.

At this point we will fix the standard uniform enumeration $W_{e,s}$
of all c.e.\ sets with the convention that at stage $s$, there is at
most one pair $e, x$ where $x$ enters $W_e$ at stage $s$.  Some
details on how we can effectively achieve this enumeration can be
found in \citet[Exercise I.3.11]{Soare:87}.  

Every c.e.\ set has an index according to this fixed enumeration.  For
the sets that we construct we have to appeal to Kleene's Recursion
Theorem to find this index. Moreover, by the standard trick of slowing
down or pausing our construction, we can assume the enumerations of
our fixed point $W_e$ and our constructed set $A$ are the same.  Our
construction, at times,
 will construct sets other than $A$.  While we
will focus on the constructed sets, the actual outcome of our
constructions will be a uniform enumeration of all constructed sets.
We will be using Kleene's Recursion Theorem with parameters to get
a function from each constructed set to an index with the same
enumeration for that set in the above enumeration.

By the Padding Lemma, we know that each c.e. set $A$ has
infinitely many indices.  By Rice's Theorem, we know that for a given
c.e.\ set $A$ the set of indices (in this fixed enumeration) for $A$ is
not computable.

Also at this point we will fix the convention that $A$, $B$, $W$, $X$
and $Y$ always refer to c.e.\ sets with some fixed index in our given
enumeration.  Now we need the following.

\begin{definition}
  A split $A_0 \sqcup A_1 = A$ is a \emph{Friedberg
    split} of $A$ iff, for all $W$, if $W-A$ is not c.e.\ then both
  $W-A_0$ and $W-A_1$ are not c.e. sets.
\end{definition}

\begin{lemma}
  If $A$ is not computable and $A_0 \sqcup A_1$ is a Friedberg split
then the split is not trivial.
\end{lemma}

\begin{proof}
  $\mathbb{N} -A$ is not c.e.\ so $\mathbb{N} -A_0$ and
  $\mathbb{N} -A_1$ are not c.e.\ and hence $A_0$ and $A_1$ are not
  computable.
\end{proof}

\begin{definition}
  For $A = W_e$ and $B=W_i$,
  $$A \backslash B = \{ x | \exists s [x \in (W_{e,s}-W_{i,s})]\}$$
  and $A\searrow B = A \backslash B \cap B$. (This is with respect to
  our given enumeration and hence this definition depends on our
  chosen enumeration.)
\end{definition}

By the above definition, $A\backslash B$ is a c.e.\ set.
$A \backslash B$ is the set of balls that enter $A$ before they enter
$B$.  If $x \in A \backslash B$ then $x$ may or may not enter $B$ and
if $x$ does enter $B$, it only does so after $x$ enters $A$ (in terms
of our enumeration). Since the intersection of two c.e.\ sets is c.e.,
$A \searrow B$ is a c.e.\ set.  The c.e.\ set $A\searrow B $ is the
c.e.\ set of balls that first enter $A$ and then enter $B$ (under the
above enumeration).  So $A \backslash B$ reads ``$A$ before $B$'' and
$A\searrow B$ reads ``$A$ before $B$ and then $B$''.

Note that for all $W$, $W\backslash A = (W-A) \sqcup (W\searrow
A)$.
Since $W\backslash A$ is a c.e.\ set, if $W-A$ is not a c.e.\ set then
$W\searrow A$ must be infinite.  (This happens for all enumerations
not just our given enumeration.)

\begin{lemma}[Friedberg]
  Assume $A= A_0 \sqcup A_1$, and, for all $e$, if $W_e \searrow A$ is
  infinite then both $W_e \searrow A_0$ and $W_e \searrow A_1$ are
  infinite.  Then $A_0 \sqcup A_1$ is a Friedberg split of $A$.
\end{lemma}
 
\begin{proof}
  Assume that $W-A$ is not a c.e.\ set but $X= W-A_0$ is a c.e.\ set.
  $X-A = (W-A_0) - A = W-A$ is not a c.e.\ set.  So $X\searrow A$ is
  infinite which implies that $X \searrow A_0$ is infinite but
  $X\searrow A_0 = (W -A_0) \searrow A_0 = \emptyset$.  Contradiction.
\end{proof}

Friedberg invented the priority method to split every c.e.\ set into
two disjoint c.e.\ sets while meeting the hypothesis of the above
lemma.

\begin{theorem}[Friedberg]\label{Fried}
  Every non-computable set $A$ has a Friedberg split.
\end{theorem}

\begin{proof}
  When a ball $x$ enters $A$ at stage $s$ we add it to one of $A_0$ or
  $A_1$ but which one $x$ enters is determined by priority.  Our
  requirements are:
  \begin{equation}\tag*{$\mathcal{P}_{e,i,k}$:}
    \label{eq:1}
    \text{if }W_e \searrow A \text{ is infinite then } |W_e \searrow A_i
    | \geq k. 
  \end{equation}
  We say $x$ meets $\mathcal{P}_{e,i,k}$ at stage $s$ if
  $|W_e \searrow A_i | = k - 1$ by stage $s-1$ and if we add $x$ to
  $A_i$ at stage $s$ then $|W_e \searrow A_i | = k$ at stage $s$.
  Find the smallest $\langle e, i, k \rangle$ that $x$ can meet and
  add $x$ to $A_i$ at stage $s$. If no such triple can be found, add
  $x$ to $A_0$ at stage $s$. It is not hard to show that all the
  $\mathcal{P}_{e,i,k}$ are met.
\end{proof}

Observe that the procedure in Theorem~\ref{Fried} is uniform. Given
this we made the following defintion and corollary.


\begin{definition}
  A computable function $h$ is a \emph{splitting procedure} iff, for
  all $e$, if $h (e) = \langle e_0, e_1 \rangle$ then
  $W_{e_0} \sqcup W_{e_1}$ is a split of $A$ and if $W_e$ is not
  computable then this split is not trivial.  If $h$ is a splitting
  procedure, we say that $h(e)$ gives a split of
  $W_e$ or splits $W_e$. 
\end{definition}

\begin{corollary}[of Friedberg's Proof] \label{uni} There is a
  splitting procedure $h$ such that if $W_e$ is not computable then
  $h(e)$ gives a Friedberg split of $W_e$.
\end{corollary}

\section{Non-Trivial non-Friedberg
    Splits}
  
The above section brings us to the following question:

\begin{question}
    When does a c.e.\ set have a non-trivial non-Friedberg split? 
  \end{question}

  This question was first asked, in a different form, as Question~1.4
  in \citet{incomputable}. In \cite{incomputable}, it was asked if
  there is a definable collection of c.e.\ sets such that for each set
  $A$ in this collection the Friedberg splits of $A$ are a proper
  subclass of the non-trivial splits of $A$.  This question later
  appeared, in yet a different form, as Question~4.6 in the first
  unpublished version of \citet*{MR3436364}.  There it was suggested
  to compare the class of all c.e.\ sets all of whose non-trivial
  splits are Friedberg with the $\mathcal{D}$-maximal sets (defined
  below).  As we will see in Theorem~\ref{Shav} every form of this
  question was answered by \citet{Shav}. Shavrukov showed that a c.e.\
  set $A$ has a non-trivial non-Friedberg split iff $A$ is not
  $\mathcal{D}$-maximal.


  \subsection{There are c.e.\ sets with non-trivial non-Friedberg
    Splits}\label{sec:proviso-2:-there}

  Let $R$ be an infinite, coinfinite, computable set. There is a
  non-computable c.e.\ subset of $R$, call this set $K_R$.  There is a
  non-computable c.e.\ subset of $\overline{R}$, call this set
  $K_{\overline{R}}$.  Let $A = K_R \sqcup K_{\overline{R}}$.  Then
  $K_R \sqcup K_{\overline{R}}$ is a non-trivial split of $A$.
  $R- A = R-K_R $ is not c.e.\ but $R-K_{\overline{R}} = R$ is a c.e.\
  set. So this split is not Friedberg.  Please note that the set $A$
  and its non-trivial non-Friedberg split are built simultaneously.

  See Theorem~\ref{Shav} for more examples of sets with non-trivial
  non-Friedberg Splits.  There are published examples of sets with
  non-trivial non-Friedberg splits.  In Section~3.2 of
  \citet*{MR3436364}, a number of such sets are constructed.  But,
  like in the construction in the above paragraph and
  Theorem~\ref{Shav}, for the examples in \cite*{MR3436364} the set
  $A$ and its non-trivial non-Friedberg split are built
  simultaneously.

\subsection{There are c.e.\ sets without non-trivial non-Friedberg
  Splits}

For this we need the following definitions:

\begin{definition}
    \begin{enumerate}
      \item $\mathcal{D}(A) = \{ B | B-A $ is a c.e.\ set$\}$.
      \item $W$ is \emph{complemented} modulo
        $\mathcal{D}(A)$ iff there is a c.e.\ $Y$ such that
        $W \cup Y \cup A = \mathbb{N}$ and $(W \cap Y) - A$ is a c.e.\
        set.
      \item $A$ is \emph{$\mathcal{D}$-hhsimple} iff, for every c.e.\
        $W$, 
        $W$ is complemented modulo $\mathcal{D}(A)$.
      \item A c.e.\ set $W$ is $0$ modulo $\mathcal{D}(A)$ iff
        $W \in \mathcal{D}(A)$.
      \item A c.e.\ set $W$ is $1$ modulo $\mathcal{D}(A)$ iff there
        is a $Y$ such that $Y \cap A =\emptyset$ and
        $W \cup Y \cup A = \mathbb{N}$. 
      \item A non-computable set $A$ is \emph{$\mathcal{D}$-maximal}
        iff for every $W$, 
        $W$ is complemented modulo $\mathcal{D}(A)$ and either $0$ or
        $1$ modulo $\mathcal{D}(A)$.
    \end{enumerate}
  \end{definition}

  Assume $W$ is $0$ modulo $\mathcal{D}(A)$.  WLOG we can assume
  $W \cap A = \emptyset$. Then $W \cup \mathbb{N} \cup A = \mathbb{N}$
  and $\mathbb{N} \cap W = W$ is disjoint from $A$. So $W$ is
  complemented modulo $\mathcal{D}(A)$.  If $W-A$ is not c.e.\ then
  $W$ is not $0$ modulo $\mathcal{D}(A)$.  A c.e.\ set $W$ is $0$
  modulo $\mathcal{D}(A)$ iff $W-A$ is a c.e.\ set. The set $W$ is $1$
  modulo $\mathcal{D}(A)$ as witnessed by $Y$ iff $W$ is complemented
  by $Y$ modulo $\mathcal{D}(A)$ and $Y$ is $0$ modulo
  $\mathcal{D}(A)$.  We will not go through the details but the
  property of a set $A$ being $\mathcal{D}$-maximal is definable in
  the c.e.\ sets, $\mathcal{E}$.

 \begin{lemma}[Cholak, Downey, Herrmann]
   All non-trivial splits of a $\mathcal{D}$-maximal set $A$ are
   Friedberg.
 \end{lemma}

 \begin{proof}

   Let $A_0 \sqcup A_1 = A$ be a non-trivial split of $A$.  Assume
   that $W -A$ is not a c.e.\ set. So $W\cup A$ is $1$ modulo
   $\mathcal{D}(A)$.  Then, for some $Y$,
   $W \cup A \cup Y = \mathbb{N}$ and $Y \cap A = \emptyset$.  If
   $W-A_0$ is c.e.\ then
   $A_0 \sqcup \big((W - A_0) \cup A_1 \cup Y\big) = \mathbb{N}$ and
   hence $A_0$ is computable.  Contradiction.
 \end{proof}

 This result and the above proof explicitly appears in an earlier
 unpublished version of \citet*{MR3436364} but not in the published
 version.  It was first implicitly mentioned in \citet*{MR1807845}.
 It follows a similar result about maximal sets in
 \citet{Downey.Stob:92}.


 \subsection{The Herrmann and Kummer Splitting Theorem}

 Shortly we will need the following theorem.\footnote{The Herrmann and
   Kummer Splitting Theorem appears, in a very different form, in
   \citet{Herrman.Kummer:94}.  This theorem appears in the only if
   direction of the proof of Theorem~2.4 of \citet{Herrman.Kummer:94}
   starting on page 63 from the first full paragraph on that page.  It
   is interesting enough to be isolated in its own right as a
   theorem.}


 \begin{theorem}[Herrmann and Kummer Splitting Theorem] 
   Let $A$ and $B$ be c.e.\ sets such that $A \subseteq B$ and $B$ is
   non-complemented modulo $\mathcal{D}(A)$.  Then there are $B_0$ and
   $B_1$ such that $B_i$ is non-complemented modulo $\mathcal{D}(A)$
   and $B_0 \sqcup B_1 = B$.
\end{theorem}

\begin{proof}
  As balls $x$ enter $B$ they will be enumerated into either $B_0$ or
  $B_1$.  So $B = B_0 \sqcup B_1$. Let $Y_e, Z_j$ be two listings of
  all c.e.\ sets. We need to meet the requirements:
  \begin{equation}\tag*{$\mathcal{R}_{e,j,i}$:}
    \text{either } B_i \cup A \cup Y_e \neq \mathbb{N} \text { or }
    (B_i \cap Y_e) - A \neq Z_j.  
  \end{equation}
  If we fail to meet this requirement then $Y_e$ and $Z_j$ witness
  that $B_i$ is complemented modulo $\mathcal{D}(A)$.

  We need a \emph{disagreement} function.  Let $l(e,j,i,s)$ be the
  least $x \leq s$ such that either
  $x \notin B_{i,s} \cup A_s \cup Y_{e,s}$, or
  $x\in ( (B_{i,s} \cap Y_{e,s}) - A_s)$ iff $x\notin Z_{j,s}$.  
  If $x$ does not exist, let $l(e,j,i,s) = s$. The $\lim_s l(e,j,i,s)$
  exists iff we will have meet $\mathcal{R}_{e,j,i}$. 

  We will use $l$ to define a \emph{restraint} function,
  $r(e,j,i,-1) = \langle e , j, i \rangle$ and $r(e,j,i,s)$ is the max
  of $r(e,j,i,s-1)$ and $l(e,j,i,s)$. Again, the $\lim_s r(e,j,i,s)$
  exists iff we will have meet $\mathcal{R}_{e,j,i}$.  Moreover
  $r(e,j,i,s)$ is a non-decreasing function in $s$.

  When a ball $x$ enters $B$ at stage $s$ find the least
  $\langle e , j, i \rangle$ such that $x \leq r(e,j,i,s)$ and add $x$
  to $B_i$.

  Let $\langle e , j, i \rangle$ be the least triple such that
  $\lim_s r(e,j,i,s)$ does not exist. Let $x$ be such that for all
  $\langle e' , j', i' \rangle < \langle e , j, i \rangle$,
  $\lim_s l(e',j',i',s) < x$. Assume $i=0$. Then $B_1$ is computable
  (for all $y > x$, after $r(e,j,0,s)>y$, $y$ cannot enter $B_1$),
  $Y_e$ and $Z_j$ witness that $B_0$ is complemented modulo
  $\mathcal{D}(A)$. Now $Y = Y_e \cap \overline{B_1}$ and $Z_j$
  witness that $B$ is complemented modulo $\mathcal{D}(A)$.
  Contradiction.  Similarly if $i = 1$.
\end{proof}

This construction is uniform.  Given an index for $B$ we can uniformly
get a split of $B$ via the above theorem.  Assume $B$ is $0$ modulo
$\mathcal{D}(A)$ witnessed by the c.e.\ set $Z = B -A$.  Then
$\mathbb{N}$ and $Z$ witness that $B$ and any splits of $B$ are
complemented modulo $\mathcal{D}(A)$.  Let $e'$ and $j'$ be the least
such that $Y_{e'} = \mathbb{N}$ and $Z_{j'} = Z$, and
$l(e',j',i,s) = s$ (this last item just takes playing a little with
the enumeration of these sets).  For some $e \leq e', j \leq j'$ and
$i$, $\lim_s r(e,j,i,s)$ does not exist and the argument above shows
that the split is trivial. So if $B \subseteq A$ this split will be
trivial.  So this theorem does not give rise to a splitting
procedure.

If $B$ is not complemented modulo $\mathcal{D}(A)$ then it is open if
the above split (as given above) is always Friedberg.  We conjecture
yes with the following evidence: We can combine the requirements
$\mathcal{P}$ from the proof of Theorem~\ref{Fried} with the one here
to force the split to be a Friedberg split.

We also want to point out that the Herrmann and Kummer Splitting
Theorem is very similar to the Owings Splitting Theorem.  $B$ is
\emph{non complemented modulo $A$} iff $B-A$ is not co-c.e.\ iff
$\overline{B} \cup A$ is not c.e.  The following theorem is an easy
corollary of the Owings Splitting Theorem, \cite{Owings:67}.  Also
see \citet[X.2.5]{Soare:87}.

\begin{theorem}[Owings]
  Let $A$ and $B$ be c.e.\ sets such that $A \subseteq B$ and $B$ is
  non-complemented modulo $A$. Then there are $B_0$ and $B_1$ such
  that $B_i$ is non-complemented modulo $A$ and $B_0 \sqcup B_1 = B$.
\end{theorem}

We are not going to provide a proof. The standard proof is
\citet[X.2.5]{Soare:87}.  What is not clear is whether this standard proof
always provides a Friedberg split and, if $B \subseteq A$, whether the
resulting split is non-trivial. We can arrange the enumeration (let
$W_0 = \mathbb{N}$) such that if $B \subseteq A$ then the resulting
split is non-trivial. But it is open what occurs when we use the standard
enumeration. So it is unknown if the Owings Splitting Theorem gives a
splitting procedure. 

The Owings and the Herrmann and Kummer Splitting theorems are like
Friedberg's in that all three are uniform, but unlike Friedberg's in
that they do not necessarily provide non-trivial splits when
possible. Herrmann and Kummer Splitting Theorem does not give rise to
a splitting procedure.  It is open if the Ownings Splitting Theorem
gives rise to a splitting procedure.  Friedberg Splitting Theorem
does give rise to a splitting procedure.

There is one more (little) known splitting theorem, \citet{MR1780060},
which extends all three of the splitting theorems above discussed in
this subsection.  Let $\mathcal{E}$ be the collection of c.e.\ sets
with inclusion, intersection, union, $\emptyset$ and $\mathbb{N}$;
this is called the lattice of c.e.\ sets.  An \emph{ideal} of
$\mathcal{E}$ is a collection of sets $\mathcal{I}$ such that
$\emptyset \in \mathcal{I}$ and $\mathcal{I}$ is closed under subset
and inclusion.  An ideal $\mathcal{I}$ is $\Sigma^0_3$ if the relation
$W_e \in \mathcal{I}$ is $\Sigma^0_3$.  $\mathcal{F}$, collection of
all finite sets, is an $\Sigma^0_3$ ideal.  For any $A$, so are
$\mathcal{S}(A) = \{ B | B \subseteq A\}$ and $\mathcal{D}(A)$.  $W$
is \emph{complemented modulo $\mathcal{I}$} iff there is a $Y$ such
that $W \cup Y = \mathbb{N}$ and $W \cap Y$ is in $\mathcal{I}$.  For
any $A$, the Friedberg, Ownings, and Herrmann and Kummer Splitting
Theorems, respectively, imply any $B$ which is non-complemented modulo
$\mathcal{F}$, $\mathcal{S}(A)$, or $\mathcal{D}(A)$ can be split into
$B_0$ and $B_1$ such that each $B_i$ is non-complemented modulo
$\mathcal{F}$, $\mathcal{S}(A)$, or $\mathcal{D}(A)$.

\begin{theorem}[\citet{MR1780060} ]
  Let $\mathcal{I}$ be any $\Sigma^0_3$ ideal.  If $B$ is
  non-complemented modulo $\mathcal{I}$ then $B$ can be split into
  $B_0$ and $B_1$ such that each $B_i$ is non-complemented modulo
  $\mathcal{I}$.
\end{theorem}

We will not include a proof here.  Unlike the other three splitting
theorems discussed here the proof is not finite injury.  It is uniform
in $\mathcal{I}$. Since $\mathcal{I}$ can equal $\mathcal{D}(A)$, it
does not always give raise to a splitting procedure. What happens when
$\mathcal{I}$ is $\mathcal{S}(A)$ is open.

\subsection{Shavrukov's Result} \label{Shavrukov}

First we need to use the Herrmann and Kummer Splitting Theorem for the
following corollary.  The proof is not uniform.

\begin{corollary}
  For all non-computable non-$\mathcal{D}$-maximal $A$, there are
  disjoint $X_0$ and $X_1$ such that $X_i -A$ is not c.e.\ and
  $A \subseteq X_0 \sqcup X_1$.
\end{corollary}

\begin{proof}
  When $A$ is not $\mathcal{D}$-hhsimple there is a c.e.\ $X$ such
  that $A \subseteq X$ and $X$ is not complemented modulo
  $\mathcal{D}(A)$.  Apply the above Herrmann and Kummer Splitting
  Theorem to get $X_0 \sqcup X_1 = X$ where the $X_i$s are also not
  complemented modulo $\mathcal{D}(A)$.  If $X_i -A$ is c.e.\ then
  $X_i$ is $0$ and hence complemented modulo $\mathcal{D}(A)$.
  Therefore $X_i -A$ is not a c.e. set.

  Otherwise $A$ is $\mathcal{D}$-hhsimple but not
  $\mathcal{D}$-maximal.  So there must be a c.e.\ superset $W$ of $A$
  which is not $0$ or $1$.  So $W-A$ is not a c.e.\ set. There is a
  $Y$ such that $W \cup Y = \mathbb{N}$, $(W \cap Y )-A$ is c.e.\ but
  $Y-A$ is not a c.e.\ set.

  Let $X_0 = W \backslash Y$ and $X_1 = Y \backslash W$.  Now
  $W= X_0 \cup ( W \cap Y)$. So
  $W-A = (X_0 -A) \cup (( W \cap Y) - A)$. The set $ ( W \cap Y) - A$
  is known to be c.e., so if $X_0 - A$ is c.e.\ then so is $W-A$.
  Therefore $X_0 -A$ is not a c.e.\ set.  $Y= X_1 \cup ( W \cap
  Y)$.
  So $Y-A = (X_1 -A) \cup (( W \cap Y) - A)$. $ ( W \cap Y) - A$ is
  known to be c.e., so if $X_1 - A$ is c.e.\ then so is $Y-A$.
  Therefore $X_1 -A$ is not a c.e.\ set.
\end{proof}

\begin{theorem}[Shavrukov]\label{Shav}
  All c.e.\ non-computable non-$\mathcal{D}$-maximal sets $A$ have
  non-trivial non-Friedberg splits.
\end{theorem}

  \begin{proof}
    By the above corollary, there are disjoint $X_0$ and $X_1$ such
    that $X_i -A$ is not c.e.\ and $A \subseteq X_0 \sqcup X_1$.  If
    $X_i \cap A$ were computable then
    $X_i -A = X_i \cap \overline{(X_i \cap A)}$ is c.e.  Therefore
    $X_0 \cap A, X_1 \cap A$ is a non-trivial split of $A$.  $X_0 - A$
    is not c.e.\ but $X_0 - (X_1 \cap A) = X_0$ is a c.e.\ set.  Hence
    $X_0 \cap A, X_1 \cap A$ is a non-trivial non-Friedberg split.
    \end{proof}

     \begin{corollary}[Shavrukov]
       All of $A$'s non-trivial splits are Friedberg iff $A$ is
       $\mathcal{D}$-maximal.
    \end{corollary}

    Again we want to thank V.\ Yu.\ Shavrukov for allowing us to
    include his results.  The proof we presented here is very
    different than Shavrukov's, see \cite{Shav}. Shavrukov's proof used the fact that
    every $\mathcal{D}$-hhsimple is not a diagonal. For the
    definition of a diagonal set see \citet{Kummer:01} and
    \citet{Herrman.Kummer:94}.

    \section{Uniform non-trivial non-Friedberg Splits}

The question we will answer in this section follows:

 \begin{question}

   Is there a splitting procedure $h$ such that all
   non-$\mathcal{D}$-maximal sets $W_e$ are split by $h(e)$ into a
   non-trivial non-Friedberg split?
  \end{question}

The answer is no by the following theorem:

\begin{theorem} \label{nonuni} For every total computable $h$ there is
  an $e$ such that $W_e$ is not computable and
  $h(e) = \langle e_0, e_1 \rangle$ then either
    \begin{enumerate}
    \item $W_{e_0}, W_{e_1}$ is not a split of $W_e$,
    \item $W_{e_0} \sqcup W_{e_1}$ is a trivial split of $W_e$, or
    \item $W_{e_0} \sqcup W_{e_1}$ is a Friedberg split of $W_e$ and
      $W_e$ is not $\mathcal{D}$-maximal.
    \end{enumerate}
    Moreover given an index for $h$ we can effectively find $e$.
  \end{theorem}

  Hence if $h$ is a splitting procedure then Case (3) applies.
  Actually, Case (3) applies infinitely often.


  \begin{corollary}\label{sec:unif-nontr-non}
    Let $h$ be a splitting procedure.
    Then there is an infinite
    set $J$ of indices that, for all $e \in J$, $W_e$ has a
    non-Friedberg split but the split given by $h(e)$ is a Friedberg
    split.
\end{corollary}

\begin{proof}[Proof of the Corollary]
  Let $h_0 =h$ and apply Theorem~\ref{nonuni} to get $e_0$. Only Case
  (3) can apply. So $W_{e_0}$ has a non-trivial non-Friedberg split
  but $h(e_0)$ gives a Friedberg split.  Inductively, assume for all
  $j \leq i$, that $h_j$ and distinct $e_j$ exist and that Case (3)
  applies to $W_{e_j}'$.  Let $W_{a_i} \sqcup W_{b_i}$ be a
  non-trivial non-Friedberg split of $W_{e_i}$.  Let
  $h_{i+1}(e_i) = \langle a_i, b_i \rangle$ and if $e \neq e_i$ let
  $h_{i+1}(e) = h_i(e)$.  Apply Theorem~\ref{nonuni} to $h_{i+1}$ to
  effectively get an $e_{i+1}$.  Case (3) applies to $e_{i+1}$ and
  $e_{i+1} \neq e_j$, for all $j \leq i$. Let $J$ be the 
  infinite set $ \{ e_i | i \in \omega\}$.
\end{proof}

We can create a splitting procedure that is correct on infinite many
indices of a non-$\mathcal{D}$-maximal set.  Take
$A_0 \sqcup A_1 = A= W_a \sqcup W_b = W_c$ to be a non-trivial
non-Friedberg splitting of $A$.  Using the padding lemma, let $I$ be
an infinite computable set of indices for $A$.  Define $h(e)$ to be
$\langle a, b \rangle$ if $e \in I$ and $h_F(e)$ otherwise, where
$h_F$ is from Corollary~\ref{uni}.  By Rice's Theorem, $I$ is not all
indices for $A$. But the following is open.

\begin{question}
  Is there a splitting procedure $h$ and a c.e.\ set $A$ with a
  non-trivial non-Friedberg split such that, 
  if $W_e = A$ then 
  $h(e)$ gives a a non-trivial non-Friedberg split of $W_e= A$?
\end{question}

\section{Proof of Theorem~\ref{nonuni}}

The goal of the rest of the paper is to provide a proof of the above
Theorem~\ref{nonuni}. Assume that we are given $h$ and we will
construct $A$.  Via the Recursion Theorem we can assume that
$W_e = A$.  Also assume that $h(e) = \langle e_0, e_1 \rangle$.

For our proof we will work using an oracle for certain $\Pi^0_2$
questions. Certainly $\mathbf{0''}$ works but is overkill. The index
set of all infinite c.e.\ sets works nicely. We will use a tree
argument to provide answers to our $\Pi^0_2$ questions. The tree will
also provide a framework for our construction.

We will build $A$ in pieces. First we will construct a $\Delta^0_3$
list of pairwise disjoint computable sets $R$ such that every c.e.\
set or it's complement will be in the union of finitely many of these
computable sets and the union of all them is $\mathbb{N}$.  Inside
each of these computable sets we will build a piece of $A$.  The
default is that $A$ will be maximal inside each $R$ but finite or
cofinite inside $R$ are also possible. The construction will ensure
that the union of these pieces is a c.e.\ set $A$.  If $A$ is maximal
in only finitely many of these computable sets then $A$ will turn out
to be $\mathcal{D}$-maximal. 

We will \emph{try} to construct infinite, coinfinite, computable sets
$R_i$ such that, for all $j$, either
\begin{equation}
  \label{eq:2}
  W_j \subseteq^* \bigsqcup_{i\leq j} R_i \cup A,
\end{equation}

or
\begin{equation}
  \label{eq:3}
  W_j \cup \bigsqcup_{i\leq j} R_i \cup A =^* \mathbb{N}.
\end{equation}
(We will remind the reader that $X =^*Y$ iff $(X-Y) \sqcup (Y-X)$ is
finite.)  Since these sets are meant to be computable we also have to
build $\overline{R}_i$ while we are building $R_i$. Assume that we
have built the sets $R_i$ up to $j$.  The balls in
$\bigcap_{i< j} \overline{R}_i$ have not yet been added to $R_j$ or
$A$.  So our construction will ensure
$(\bigcap_{i< j} \overline{R}_i ) = (\bigcap_{i< j} \overline{R}_i )
\backslash A$ is infinite.  To build $R_j$ ask if
\begin{equation}
  \label{eq:5}
  P_j = (W_j \cap \bigcap_{i< j} \overline{R}_i ) \backslash
  A
\end{equation}
is infinite.  This is a $\Pi^0_2$ question.  If $P_j$ is infinite, we
will build $\overline{R}_j$ as a subset of $W_j$, so that
Equation~\ref{eq:3} is satisfied. When we add balls from the set
$\bigcap_{i< j} \overline{R}_i$ to $R_j$, we will make sure that there
is at least one ball in $W_j \cap \bigcap_{i< j} \overline{R}_i$
currently uncommitted. We will add that ball to $\overline{R}_j$ and
the rest of the balls under consideration to $R_j$.  We will do this
infinitely often. In this case, we satisfy Equation~\ref{eq:3}.  If
$P_j$ is finite, then, since
$(\bigcap_{i< j} \overline{R}_i ) \backslash A$ is infinite, we just
build $R_j$ and $\overline{R}_j$ to be infinite within
$\bigcap_{i< j} \overline{R}_i$ and Equation~\ref{eq:2} is satisfied.

Now inside each $R_i$ we will build $A$ to be finite, cofinite, or
maximal depending on various outcomes. The default will be for $A$ to be 
maximal in $R_i$. To do this we use the construction presented in
\citet[X.3.3]{Soare:87} as a guide to work inside $R_i$. We will go
over the details later. Since maximal sets are not computable, $A$
will not be computable.  Assume that $A$ is maximal inside $R_i$ and
$R_l$, where $l \neq i$, then, since $A \cap R_l$ is a non-computable
subset of $A \cap \overline{R}_i$,
$A = (A \cap R_i) \sqcup (A \cap \overline{R}_i)$ is a non-trivial
non-Friedberg split of $A$.  The details follow the construction in
Subsection~\ref{sec:proviso-2:-there}. Now by Theorem~\ref{Shav}, $A$
is not $\mathcal{D}$-maximal.  If $W_{e_0} \sqcup W_{e_1}$ is not a
split of $A$ then we are done.  So we may as well assume that
$W_{e_0} \sqcup W_{e_1}$ is a split of $A$.

We will now consider how this split behaves inside each $R_i$.  Since
$A$ is maximal inside $R_i$ there are two choices either the split is
trivial or Friedberg. We are going to ask an infinite series of
questions designed to tell if the split inside $R_i$ is trivial. The
questions are is ``$W_k \sqcup (W_{e_0} \cap R_i )= R_i$" and is
``$W_k \sqcup (W_{e_1} \cap R_i )= R_i$", for all $k$.  Again these
questions are $\Pi^0_2$.  A positive answer will tell us the split is
trivial inside $R_i$ and which set $W_{e_0} \cap R_i$ or
$W_{e_1} \cap R_i$ is computable.

Assume that we get a positive answer and the information that the set
$W_{e_0} \cap R_i$ is computable.  In this case we will take the
following action: Dump almost all of $R_l$, for $l < i$, into $A$ and,
for $l> i$, stop adding balls from $R_l$ into $A$.  In fact, stop
building $R_l$.  In this case, $A$ is computable outside $R_i$ and
hence $W_{e_0}$ must also be computable. So $W_{e_0} \sqcup W_{e_1}$
is a trivial split of $A$.  We act similarly if $W_{e_1} \cap R_i$ is
computable.

If none of the answers to these questions for each $R_i$ is positive
then $W_{e_0} \sqcup W_{e_1}$ is a non-trivial split of $A$. We know
inside each $R_i$ the split is Friedberg. We must show that globally
the split is Friedberg.  Let's consider $W_j$.  If Equation~\ref{eq:2}
holds, then $W_j - A \subseteq^* \bigsqcup_{i\leq j} R_i$. So, if
$W_j - A$ is not a c.e.\ set neither are $W_j - W_{e_0}$ and
$W_j - W_{e_1}$.  So assume Equation~\ref{eq:3} holds, $W_j - A$ is
not a c.e.\ set, but $W_j - W_{e_0}$ is a c.e.\ set.  For any $n>j$,
$(W_j - A) \cap R_n$ cannot be a c.e.\ set. But
$(W_j - W_{e_0}) \cap R_n$ is a c.e.\ set. This contradicts that our
split is Friedberg inside $R_n$.  A similar argument works if
Equation~\ref{eq:3} holds, $W_j - A$ is not a c.e.\ set, but
$W_j - W_{e_1}$ is a c.e.\ set.  Our split is a Friedberg split.

With one positive answer, we must take action to ensure that our given
split is trivial.  One positive answer is a $\Sigma^0_3$ event. If all
questions have negative answers then we have a $\Pi^0_3$ event and, in
this case, our split is a Friedberg split.

\subsection{Coding our $\Pi^0_2$ Questions via a Tree}

We will work with the tree, $2^{< \omega}$.  We consider the tree to
grow downward.  At the empty node, $\lambda$, we will construct $A$
and $\overline{R}_\lambda= \tilde{R}_\lambda = \mathbb{N}$.  At nodes
$\alpha$ of length $i^2 >0$ we will construct $R_\alpha$ and
$\tilde{R}_\alpha$ ($\overline{R}_\alpha = \sqcup_{\beta \subset \alpha} R_\beta \sqcup
\tilde{R}_\alpha$.)
We will call such nodes $R$-nodes.  The idea is that if $f$ is the
true path, $|\alpha| = i^2$, and $\alpha \prec f $, then
$R_\alpha = R_i$ and
$\bigsqcup_{\beta\subseteq \alpha} R_\beta \sqcup \tilde{R}_\alpha =^*
\mathbb{N}$.  (We will start indexing the $R_i$ at $1$.)

Since we need to ask questions about the potential $R_i$'s we need the
indices for the $R_\alpha$'s.  So the real outcome of our construction
is a pair of functions $g$ and $\tilde{g}$ such that $W_{g(\lambda)} =
A$, $W_{g(\alpha)} = R_\alpha$, and $W_{\tilde{g}(\alpha)} =
\tilde{R}_\alpha$, for all $\alpha$.  Via the Recursion Theorem, we
can assume we know $g$ and $\tilde{g}$ prior to the construction. We
will use this knowledge to code our questions into the tree.

Let $|\gamma| = j^2-1$.  Let $\delta \subset \gamma$ such that
$|\delta| = (j-1)^2$. At $\gamma$ we will code the question ``Is $(W_j
\cap \tilde{R}_\delta ) \backslash A$ infinite?''.  \emph{Strictly}
between two $R$-nodes of length $j^2$ and $(j+1)^2$ there are
$\big((j+1)^2-1\big)- j^2 = 2j$ nodes.  If $|\gamma | = j^2 + 2k-2$,
$1 \leq k \leq j$, $\beta \preceq \gamma$, and $|\beta| = k^2$, then
at $\gamma$ code the question ``Does $W_j \sqcup (W_{e_0} \cap R_\beta
)= R_\beta$?''.  If $|\gamma | = j^2 + 2k-1$, $1 \leq k \leq j$,
$\beta \preceq \gamma$, and $|\beta| = k^2$, then at $\gamma$ code the
question ``Does $W_j \sqcup (W_{e_1} \cap R_\beta )= R_\beta$?''.
(The only difference in these two sentences is the length of $\gamma$
differences by 1 and the second uses $W_{e_1}$ rather than $W_{e_0}$.)

Via the use of the Recursion Theorem, as we discussed two paragraphs
above, these are uniformly $\Pi^0_2$ questions.  There is a uniform
reduction from these questions to the index set of infinite c.e. sets
or INF.  So uniformly, for all $\gamma$, we can associate a c.e.\
\emph{chip} set $C_\gamma$ such that $C_\gamma$ is infinite iff the
question coded at $\gamma$ has a positive answer.

Earlier we have called some nodes $R$-nodes.  These were the nodes
whose length is a prefect square.  Other than the empty node, we will
call the remaining nodes $A$-nodes; they provide answers to questions
coded at $\alpha$'s predecessor, $\alpha^-=\gamma$.  We call an
$A$-node $\alpha$ \emph{positive} iff $\alpha\hat{~}1 = \gamma$.
Otherwise an $A$-node is negative.

We will inductively define the \emph{true path}, $f$. $\lambda$ is on
$f$.  Assume that $\alpha \preceq f$. If $\alpha$ is a positive
$A$-node then $f = \alpha$. Otherwise, $\alpha\hat{~}1 \prec f $ iff
$C_\alpha$ is infinite and $\alpha\hat{~}0 \prec f $ iff $C_\alpha$ is
finite.  Since nodes of length $0$ and $1$ are not $A$-nodes, there is
always an $R$-node on the true path. Either all the $A$-nodes on $f$
are negative or $f$ is finite and ends with a positive $A$-node.

A key to the construction is the approximation to the true path at
stage $s$, $f_s$.  Define $f_0 = \lambda$, the empty node. Assume that
$\alpha \subseteq f_{s+1}$ and $|\alpha| < s^2$.  If $\alpha$ is a
positive $A$-node, let $f_{s+1} = \alpha$. Assume that $\alpha$ is not
a positive $A$-node. Let $t$ be the greatest stage less than $s+1$
such that $\alpha \subseteq f_t$.  If no such stage exists, let $t=0$.
If $C_{\alpha, t} \neq C_{\alpha,s+1}$ then let $\alpha\hat{~}1
\subseteq f_{s+1}$.  Otherwise, $\alpha\hat{~}0 \subseteq f_{s+1}$.

Since nodes of length $s^2$ are $R$-nodes, for $s>0$, $f_s$ always
ends in an $R$-node or a positive $A$-node.  We say $\alpha <_L \beta$
(or $\alpha$ is to the left of $\beta$) iff $\alpha\subsetneq \beta$
or there is a $\gamma$ such that $\gamma\hat{~}1 \subseteq \alpha$ and
$\gamma\hat{~}0 \subseteq \beta$. By induction on $l$, we can show
that $\liminf_s f_s \restriction l = f \restriction l$ (the $\liminf$
is measured w.r.t.\ $<_L$).  So, $\liminf_s f_s = f $. If
$f_s <_L \alpha$ then there is always a least (in terms of length)
$R$-node or positive $A$-node, $\beta$, such that
$\beta \subseteq f_s$ and $\beta <_L \alpha$.

\subsection{Action on the Tree}\label{action}

We will use the tree and $f_s$ to construct $A$, $R_\alpha$, and
$\tilde{R}_\alpha$, for all $\alpha$.  We think of this construction
as a pinball machine. Integers or balls enter at top node, $\lambda$,
and move downwards and leftwards. The position of a ball, $x$, at the
end of stage $s$ is given by the function $\alpha(x,s)$.  The movement
on the tree is done such that the $\lim_s \alpha(x,s)$ exists. Let
$\alpha(x)=\lim_s \alpha(x,s)$. Initially, $\alpha(x,s)$ is not
defined (so $x$ is not on the machine) and, unless explicitly changed,
$\alpha(x,s)$ remains the same from stage to stage.  For the balls on
the machine, at every stage $s$, $\alpha(x,s)$ is always an $R$-node
or a positive $A$-node, and $|\alpha(x,s)| \leq x^2$.  (The bound
$x^2$ was chosen since balls can only rest at $R$-nodes or positive
$A$ nodes and the length of $R$-nodes are perfect squares.)  If a ball
$x$ enters $A$ at stage $s$, $x$ is removed from the tree at stage $s$
and $\alpha(x,s)$ is undefined again.

Entering the machine and leftward movement is determined by $f_{s+1}$.
Downward movement will be discussed later.  Let $\beta$ be the
$R$-node of length $1$ such that $\beta \subseteq f_{s+1}$.  Let
$\alpha(s,s+1) = \beta$.  So all the balls on the machine at stage $s$
are less than $s$.  Assume that $\alpha(x,s) = \alpha$ and
$f_{s+1} <_L \alpha$.  Then there is always a least (in terms of
length) $R$-node or positive $A$-node, $\beta$, such that
$\beta \subseteq f_{s+1}$, $\beta \not\subset \alpha$, and
$\beta <_L \alpha$.  Let $\alpha(x,s+1) =\beta$.  Since
$|\alpha| \leq x^2+1$, the same is true for $\beta$. A ball $x$ can
only move leftward finitely many times.  Since $\liminf_s f_s =f$,
$\alpha(x) <_L f$ or $\alpha(x) \subseteq f$.

Assume that $\alpha$ is an $R$-node.  So the length of $\alpha$ is
$j^2$ for some $j$. Either $\alpha = \alpha^-\hat{~}1 $ or
$\alpha = \alpha^-\hat{~}0$.  At $\gamma=\alpha^-$ we asked the
question ``Is $(W_j \cap \tilde{R}_\delta) \backslash A $ infinite?'',
where $\delta$ is the greatest proper $R$-subnode of $\gamma$.
If $\alpha$ ends with a $1$, then $\alpha$ believes this set is
infinite. If $\alpha$ ends with a $0$ then $\alpha$ believes this set
is finite.  If $\alpha$ ends with a $1$ let
$P_\alpha = (W_j \cap \tilde{R}_\delta) \backslash A $.  Otherwise,
let $P_\alpha = \tilde{R}_\delta\backslash A $.  We also defined
$P_\alpha$ for positive $A$-nodes to be
$P_\alpha = \tilde{R}_\delta \backslash A $, where
$\delta\subset \alpha$ is the greatest $R$-node contained in $\alpha$.
$\alpha$ wants all balls in $P_\alpha$ to go though $\alpha$.
Moreover the construction of $A$ inside $R_\alpha$ requires that
$\alpha$ see fresh balls in $P_\alpha$. So these $\alpha$ are allowed
to pull balls in $P_\alpha$.

We will now work on the remaining movement, pulling, on our pinball
machine.  An $R$-node or positive $A$-node $\alpha$ is allowed to pull
balls from subnodes of $\alpha$ or nodes to the right of $\alpha$.
Pulling will be downward or leftward movement. The only downward
movement allowed is done via pulling.  When $\alpha$ can pull balls is
controlled by $f_s$.  When $\alpha \subseteq f_{s}$, $\alpha$ puts a
request coded by $s$ on a list denoted by $\mathcal{P}_\alpha$ at stage $s$.
$\alpha$ can only pull balls when there is an unfulfilled request on
the list. If $\alpha$ takes action (as described below) at stage $s$
then the least request on $\mathcal{P}_\alpha$ has been fulfilled.  If
$f_s <_L \alpha$ then all the current requests at stage $s$ on
$\mathcal{P}_\alpha$ are declared fulfilled.

Let $\alpha$ be an $R$-node or $A$-node of length $l$ and assume that
there is an unfulfilled request on $\mathcal{P}_\alpha$ at stage $s$.
Assume that there are two different balls, $x_0$ and $x_1$, such that
$x_i > l$, $x_i \in P_{\alpha,s}$, and either
$\alpha <_L \alpha(x_i,s)$ ($x_i$ is to the right of $\alpha$) or
$ \alpha(x_i,s) \subset \alpha$ ($x_i$ is above $\alpha$).  For
leftmost $\alpha$ and the least such pair, at stage $s+1$, take the
following action: Let $\alpha(x_i,s+1) = \alpha$ and, if $\alpha$ is a
$R$-node, then put $x_0$ into $R_{\alpha,s+1}$ and put $x_1$ into
$\tilde{R}_{\alpha,s+1}$.  For all balls $y$, such that
$|\alpha|^2 < y < \max x_i$, $y \in \tilde{R}_{\delta,s}$ (using the
above notation for $\delta$), and either $\alpha <_L \alpha(x_i,s)$ or
$ \alpha(x_i,s) \subset \alpha$, let $\alpha(y,s+1)= \alpha$ and, if
$\alpha$ is $R$-node, then add $y$ to $R_{\alpha,s+1}$. This request
is  declared fulfilled.  

There is just a little more to the construction of $R_\alpha$.  In the
next section we will discuss the construction of $A$ inside
$R_\alpha\backslash A$. Recall earlier that we said that if a ball
enters $A$ it is removed from the machine.  That means that none of
the above balls added to $R_\alpha$ and $\tilde{R}_\alpha$ are in $A$.
To make sure that $R_\alpha$ is computable when $\alpha \subset f$ we
must be sure that almost all balls from $\tilde{R}_{\delta}$ enter
$R_\alpha$ or $\tilde{R}_\alpha$.  Because of the construction to the
right of the true path, infinitely many balls in $\tilde{R}_{\delta}$
might enter $A$ before they enter $R_\alpha$ or $\tilde{R}_\alpha$.
The balls we are talking about are in the c.e.\ set $(
\tilde{R}_{\delta} \searrow A ) \backslash (R_\alpha \sqcup
\tilde{R}_\alpha)$. The above action cannot add these balls to
$R_\alpha$ or $\tilde{R}_\alpha$.  So we will simply add these balls
to $R_\alpha$. So the above set is equal to $A \searrow R_\alpha$.

Let's see inductively that for $\alpha \subset f$ and $\alpha$ is an
$R$-node, that $R_\alpha\backslash A$ is infinite,
$ \tilde{R}_\alpha \backslash A$ is infinite,
$\bigsqcup_{\beta \subseteq \alpha}R_\beta \sqcup \tilde{R}_\alpha =^*
\mathbb{N}$,
and $A \searrow R_\alpha$ is computable. Let $\delta$ be the greatest
proper $R$-subnode of $\alpha$.  If no such node exists let
$\delta=\lambda$.  So by our inductive hypothesis
$ \tilde{R}_\delta \backslash A$ is infinite.  Moreover, by the
movement on the tree, only finite many of these balls are ever to the
left of $\alpha$. Ignore those balls. Since $\alpha$ is on the true
path, infinitely many requests are placed on $\mathcal{P}_\alpha$ and
only finitely many of them are fulfilled because $f_s <_L \alpha$. We
claim all of the remaining requests are fulfilled. If not then all but
finitely balls of $P_\alpha$ can be pulled by $\alpha$ and $\alpha$
will eventually pull two balls fulfilling the desired request.  So the
action discussed two paragraphs above occurs infinitely often. We have
ensured that $R_\alpha \backslash A$ and
$\tilde{R}_\alpha\backslash A$ are infinite.  The sets $R_\beta$, for
$\beta \subseteq \alpha$, and $\tilde{R}_\alpha$ are all pairwise
disjoint.  By the action in the above paragraph the union of all these
sets is almost everything.  Since the disjoint union of these sets is
almost everything, we also have that $A\searrow R_\alpha$ is
computable. Moreover if $\alpha$ ends with a $1$ then
$\tilde{R}_\alpha \subseteq W_j$, where $|\alpha| = j^2$, and hence
$W_j \cup A \cup \bigsqcup_{\beta\subseteq \alpha} R_\beta=^*
\mathbb{N}$
and Equation~\ref{eq:3} holds.  If $\alpha$ ends with a $0$ then
$W_j \subseteq^* A \cup \bigsqcup_{\beta\subseteq \alpha} R_\beta$ and
Equation~\ref{eq:2} holds.

Assume $f$ is finite. So $\alpha=f$ is a positive $A$ node. Let
$\gamma$ be the greatest $R$-subnode of $\alpha$. Let $Z$ be the set of
$x$ such that there is a stage $s$ where $\alpha(x,s)=\alpha$.  $Z$ is
a c.e.\ set.  Because $\alpha =f$ for almost all balls $x$ in $Z$,
$\alpha(x) = \alpha$. Almost all of the balls in $Z$ never enter
$A$. $Z$ is the end of the line. Recall that $P_\alpha =
\tilde{R}_\delta \backslash A$.  By the pulling action almost all of
the balls in $P_\alpha$ will enter $Z$.  By the above paragraph,
$\bigsqcup_{\beta \subseteq \delta} R_\beta \sqcup \tilde{R}_\delta
=^* \mathbb{N}$.  So $Z$ and $A \searrow \tilde{R}_\delta$ are computable
sets.

\subsection{The construction of $A$}

We will build $A$ to be maximal inside $R_\alpha\backslash A$.  Since
$\alpha \subset f$, $R_\alpha \backslash A$ is an infinite computable
set.  Let $R= R_\alpha \backslash A$.  We build $A \cap R$ stagewise
based on the construction of a maximal set from \citet[Theorem
X.3.3]{Soare:87}.

The main requirement is to ensure that, for all $e$,
\begin{equation}
  \label{eq:4}\tag*{$\mathcal{M}_e:$}
  W_e \cap
  R \subseteq^* A \cap R \text{ or } (W_e \cap R) \cup (A \cap R) = R.
\end{equation}
$\sigma(e,x,s) = \{ i : i \leq e \wedge x \in W_{i,s}\}$ is the
$e$-state of $x$ at stage $s$.  We will have a series of markers
$\Gamma^\alpha_n$ with $a^{\alpha,s}_n$ denoting the position of
$\Gamma^\alpha_n$ at stage $s$ and such that $\overline{A}_s \cap R =
\{ a^{\alpha,s}_0 < a^{\alpha,s}_1 \ldots \}$.  Each marker $\Gamma_e$
wants to move to maximize the $e$-state of $a^\alpha_e = \lim
a^{\alpha,s}_e$.

Initially, we let $A_0 \cap R = \emptyset$ and define the
$a^{\alpha,0}_n$ accordingly.  At stage $s+1$, if there is a least $e$
such that for some least $i$, $e < i < s$ and
$\sigma(e,a^{\alpha,s}_i,s) > \sigma(e,a^{\alpha,s}_e,s)$, then we
dump $a^{\alpha,s}_e, a^{\alpha,s}_{e+1} , \ldots a^{\alpha,s}_{i-1}$
into $A$ at stage $s+1$.  So $a^{\alpha,s+1}_e = a^{\alpha,s}_i$.
Let's call this dumping the \emph{original dumping}. If $e$ does not
exist do nothing.

Certain positive $A$-nodes $\gamma$ below $\alpha$ can also dump balls
from $R = R_\alpha \backslash A$ into $A$. Let $\gamma$ be a positive
$A$-node such that $\alpha \subset \gamma$ and at $\gamma^-$ is coded
the question ``Is $W_j \sqcup (W_{e_0} \cap R_\beta) = R_\beta$
infinite'?''  or ``Is $W_j \sqcup (W_{e_1} \cap R_\beta) = R_\beta$
infinite?'', for some $j$ and some $\beta\neq \alpha$.  $\gamma$
believes that our split is trivial inside some $R_\beta$ and wants to
dump almost all of $R_\alpha$ into $A$.  Let $t_{\gamma,s}$ be the
maximum of $|\gamma|$ and the greatest stage $t$ such that $t \leq s$
and $f_t <_L \gamma$.  Assume $\gamma \subset f_{s+1}$, dump
$a^{\alpha,t_{\gamma,s}}_s$ into $A$ at stage $s+1$ (if the above
movement of balls at stage $s+1$ has already forced
$a^{\alpha,t_{\gamma,s}}_s\neq a^{\alpha,t_{\gamma,s}}_{s+1}$ that is
enough).  Let's call this dumping, \emph{extra dumping}.

The positive $A$-nodes to the left or to the right of the true path
only dump $a^{\alpha,e}_s$ finitely often.  The ones to the left of
the true path are only on $f_s$ finitely often and hence only dump
finitely many balls from $R_\alpha \backslash A$ into $A$. If $f <_L
\gamma$ then $\lim_s t_{\gamma,s}$ goes to infinity and $\gamma$ can
only dump each $a^{\alpha,e}_s$ into $A$ finitely often.

Assume $\gamma =f_s $ is a positive $A$-node and $\alpha \subset
\gamma$ and at $\gamma^-$ is coded the question ``Is $W_j \sqcup
(W_{e_0} \cap R_\beta) = R_\beta$ infinite?'' or ``Is $W_j \sqcup
(W_{e_1} \cap R_\beta) = R_\beta$ infinite?'', for some $j$ and some
$\beta\neq \alpha$. Then $\lim_s t_{\gamma,s}$ exists and almost all
balls in $R_\alpha$ are dumped into $A$, i.e.\ $(R_\alpha \backslash
A) \subseteq^* A$.  For the rest of this section we will assume the
above assumption is false.

So the extra dumping at most dumps each $a^{\alpha,e}_s$ into $A$
finitely often.  Assume that $a^{\alpha,e}_s$ will not be dumped after
stage $s$ via our extra dumping.  Since the original dumping only
dumps to increase the $e$-state and there are $2^e$ many $e$-states,
the original dumping only dumps $a^{\alpha,e}_s$ finitely often.
Hence $\lim _s a^{\alpha,e}_s$ exists and equals $ a^{\alpha,e}$.

Now we are in a position to show that the requirements $\mathcal{M}_e$
are met. Assume that $\mathcal{M}_i$ holds for $i<e$ and there is an
$(e-1)$-state $\tau$ such that almost all of $R-A$ are in state
$\tau$.  Assume all balls greater than $k$ in $R-A$ are in state
$\tau$.  Let $$M = \{ x : \exists s, n [\sigma(e-1,x,s)=\tau \wedge
n\geq k \wedge x = a^{\alpha,n}_s]\}.$$ So $R-A \subseteq^* M$. Assume
$(M \cap W_e) \backslash A$ is finite.  Then $W_e \cap R \subseteq^* A
\cap R$ and almost all balls in $R-A$ are in $e$-state $\tau$. Now
assume $(M \cap W_e) \backslash A$ is infinite.  Let $n\geq k$ and
$\sigma(e,a^{\alpha,n}_s,s) = \tau$.  Since eventually there will be
an $m$ and stage $t$ where $\sigma(e,a^{\alpha,m}_t,t) = \tau \cup
\{e\}$, $a^{\alpha,n} \neq a^{\alpha,n}_s$. So $R-A \subseteq^*
W_e$. So, $A$ is maximal inside $R_\alpha$.

\subsection{Putting it all together}

Recall that if $\alpha=f$ is a positive $A$-node then, for some $j$,
some $i$, and some $\beta \subset f$, $W_j$ witnesses that $W_{e_i}
\cap A$ is a computable subset of $R_\beta$. In this case, a
$\Sigma^0_3$ event occurs.  By the work in the above paragraph, we
know that $A \cap R_\beta$ is maximal in $R_\beta$ and hence $A \cap
R_\beta$ is not computable.  So $A$ is not computable.  So if $W_{e_0}
\sqcup W_{e_1}$ is not a split of $A$ we are done. Assume otherwise.
By our assumption, inside $R_\beta$, $W_{e_0} \sqcup W_{e_1}$ is the
trivial split. Let $\delta$ be the greatest $R$-subnode of $\alpha$.  By
work in the last paragraph of Section~\ref{action}, there is a set
$Z$, such that $\bigsqcup_{\gamma \subseteq \delta} R_\gamma \sqcup
(A\searrow \tilde{R}_\delta) \sqcup Z =^* \mathbb{N}$ and $Z \searrow A =
\emptyset$.  Now, by the above section, for all $\gamma$, such that
$\gamma\subset \alpha$ and $\gamma \neq \beta$, $A \cap R_\gamma =^*
R_\gamma$.  Therefore outside of $R_\beta$, $A$ is computable; i.e.\
$A \cap \overline{R_\beta}$ is computable.  Any split of a computable
set is trivial.  Therefore, $W_{e_0} \sqcup W_{e_1}$ is a trivial split
of $A$. So, by Theorem~\ref{nonuni}, $A$ is 
$\mathcal{D}$-maximal. 

For the remaining part of this paper, assume that $f$ is an infinite
path through $2^{<\omega}$.  So a $\Pi^0_3$ event occurs. In this case,
by the above section, for all $\alpha \subset f$, where $\alpha$ is an
$R$-node, $A \cap R_\alpha$ is not computable.  So $A$ is not
computable.  If $W_{e_0} \sqcup W_{e_1}$ is not a split of $A$ we are
done. Assume otherwise.  Let $\alpha$ be any $R$-node where $\alpha
\subset f$.  Let $A = (A \cap R_\alpha) \sqcup (A \cap
\overline{R_\alpha})$.  There is an $R$-node $\beta \neq \alpha$ on the
true path. $A \cap R_\beta$ is also not computable.  Hence, $ (A \cap
\overline{R_\alpha})$ is not computable and $A = (A \cap R_\alpha)
\sqcup (A \cap \overline{R_\alpha})$ is a non-trivial split of $A$.
The split is not Friedberg, since $R_\alpha - A$ is not a c.e.\ set
but $R_\alpha - (A \cap \overline{R_\alpha})=R_\alpha$ is computable.
Therefore, by Theorem~\ref{nonuni}, $A$ is not
$\mathcal{D}$-maximal. 

It just remains to show that $W_{e_0} \sqcup W_{e_1}$ is a Friedberg
split of $A$. We know that for all $\gamma \subset f$,
$W_{e_0} \sqcup W_{e_1}$ is a Friedberg split of $A$ inside
$R_\gamma$.  Since splits of maximal sets are either trivial or
Friedberg, otherwise the above $\Sigma^0_3$ event occurs.  We must
show that globally the split is Friedberg.  Let's consider $W_j$.  Let
$\alpha \subset f $ such that $|\alpha|=j^2$.  By the work in the
second to last paragraph of \ref{action}, either
$W_j \subseteq^* A \cup \bigsqcup_{\beta \subseteq \alpha} R_\beta$ or
$W_j \cup A \cup \bigsqcup_{\beta\subseteq \alpha} R_\beta=^*
\mathbb{N}$.
In the first case, if $W_j - A$ is not a c.e.\ set neither are
$W_j - W_{e_0}$ and $W_j - W_{e_1}$. Assume that
$W_j \cup A \cup \bigsqcup_{\beta\subseteq \alpha} R_\beta=^*
\mathbb{N}$.
Furthermore, assume $W_j - A$ is not a c.e.\ set, but $W_j - W_{e_0}$
is a c.e.\ set. For any $\gamma$, where
$\alpha \subset \gamma \subset f$, $(W_j - A) \cap R_\gamma$ cannot be
a c.e.\ set since this set contains $R_\gamma-A$ and $A$ is maximal
inside $R_\gamma$. But, since $W_j - W_{e_0}$ is a c.e.\ set,
$(W_j - W_{e_0}) \cap R_\gamma$ is a c.e.\ set. This contradicts the
fact that our split is Friedberg inside $R_\gamma$. A similar argument
works if $W_j - A$ is not a c.e.\ set, but $W_j - W_{e_1}$ is a c.e.\
set. So our split is a Friedberg split.  \hfill $\square$


\end{document}